\documentclass[11pt, oneside]{article}  
\usepackage[utf8x]{inputenc} 	
\usepackage{geometry}                		
\usepackage{graphicx}				
\usepackage{amssymb}
\usepackage{amsmath, amsthm}
\usepackage{authblk}
\usepackage[misc]{ifsym}

\newtheorem{thm}{Theorem}[section]
\newtheorem{lem}[thm]{Lemma}
\newtheorem{cor}[thm]{Corollary}
\theoremstyle{definition} 
\newtheorem{defi}[thm]{Definition}
\newtheorem{exx}[thm]{Example} 
\newenvironment{ex}
{\pushQED{\qed}\exx}
{\popQED\endexx}

\newcommand{\st}{\operatorname{st}}

\title{Gorenstein rings generated by strongly stable sets of quadratic monomials}
\author[1]{Ralf Fr\"oberg}
\author[2]{Lisa Nicklasson}
\affil[1,2]{Stockholm University}
\affil[2]{Universit\`a di Genova}

\affil[2]{\small \tt \Letter \  nicklasson@dima.unige.it}
\date{}

\begin{document}
\maketitle

\begin{abstract}
	\noindent We characterize all Gorenstein rings generated by strongly stable sets of monomials of degree two. We compute their Hilbert series in several cases, which also provides an answer to a question by Migliore and Nagel \cite{mi-na}.
\end{abstract}

\section{Introduction}
Strongly stable sets of monomials are an important tool in commutative algebra, and provides a link to combinatorics. One reason for studying strongly stable sets of monomials is the following.
When studying graded ideals of $K[x_1,\ldots,x_n]$, a much used technique is to make a general 
change of coordinates and then determine the initial ideal of the transformed ideal. This operation preserves many algebraic properties of the original ideal. It is well known \cite{ga}, that such initial ideals are Borel fixed. 
In characteristic zero being Borel fixed is the same as being strongly stable. (In positive characteristic strongly stable
implies Borel fixed only.) 

In the context of Hilbert schemes, it is known that each component and each intersection 
of components contains at least one point corresponding to a scheme defined by a Borel-fixed ideal,
and these ideals  can be used to understand its local structure, see \cite{no-sp}, \cite{le}.

Strongly stable sets also play an important role in the algebraic theory of shifting, see 
\cite[Chapter 11]{he-hi}.

In \cite{bo-co} Boij and Conca study subrings $K[f_1,\ldots,f_r]$ of $K[x_1,\ldots,x_n]$, with $f_i$ of degree $d$, and are interested in, given $n,r,d$, how to choose the $f_i$'s so that the subring they generate has minimal Hilbert series. They show that the 
$f_i$'s should constitute a strongly stable set of monomials, but it is not clear which strongly stable sets that occur. The second author of this paper made a thorough investigation of the case $d=2$ \cite{ni}. A subring generated by a strongly stable set of quadratic monomials can be realized as a quotient by a polynomial ring and a determinantal ideal. This connection is also studied in \cite{co-na-pe-yu}. A nice feature of these rings is the combinatorial interpretation of their Hilbert series. 

Rings generated by strongly stable sets of monomials is an interesting topic in itself, and was studied by De Negri in \cite{dene}. In this paper we continue the study of subrings generated by strongly stable sets of quadratic monomials, asking when such a ring is Gorenstein. Our result is a complete characterization of which strongly stable sets in degree two that give Gorenstein rings. We also provide explicit expressions of their Hilbert series in several cases. We find that Gorenstein rings are rather ubiquitous in our situation. Among other things we find lots of more Hilbert functions of Gorenstein ideals generated by quadrics than the ones given in \cite{mi-na}.

\section{Preliminaries}
Let $K$ be a field, and let $R = K[x_1,\ldots,x_n]$ be the standard graded polynomial ring in $n$
 variables. Let $R_d$ denote the $K$-space of homogeneous polynomials of degree $d$ in $R$. 
 For a linearly independent subset $W\subseteq R_d$, let $K[W]\subseteq R$ be the subring of $R$
 generated by the elements in $W$. Define the Hilbert function of such an algebra $K[W]$ as 
 HF$(K[W ], i) = \dim_K({\rm span} W^i)$, and the Hilbert series of $K[W]$ to be
 $\sum_{i=0}^\infty {\rm HF}(K[W],i)t^i$.
  
 \begin{defi} A set $W$ of monomials in $R_d$ is called strongly stable if $m \in W$ and $x_i|m$
  implies $(x_j/x_i)m \in W$ for all $j < i$.
 \end{defi}
 We use the notation st$(m_1,\ldots, m_s)$ for the smallest strongly stable set containing the monomials $m_1,\ldots,m_s$, and we say that $m_1,\ldots,m_s$ are strongly stable generators of this set. We are interested in characterizing Gorenstein rings $K[W]$ in the case when $W$ is a strongly
stable set of quadratic monomials.
Strongly stable sets of quadratic monomials can be illustrated by a shifted Ferrers diagram, as in Figure \ref{fig:firstex}. The diagram is defined as follows. The box in row $i$ and column $j$ represents the monomial $x_ix_j$. Since $x_ix_j = x_jx_i$ we only need to consider boxes on and above the diagonal in the diagram. That the set of monomials is strongly stable means precisely that if the $x_ix_j$-box is included in the diagram, so is everything above and to the left of it. 

\begin{figure}[ht]
	\centering 
	\includegraphics[scale=1.43]{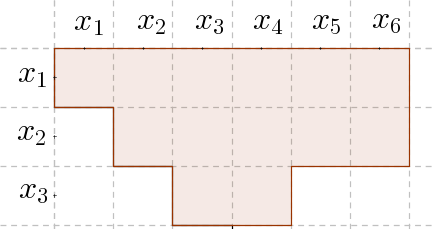}
	\caption{$\st(x_3x_4,x_2x_6)$}
	\label{fig:firstex}
\end{figure}

In searching for Gorenstein rings we will use the following theorem.

\begin{thm}\cite[Theorem 4.4]{{st}}
	If a graded Cohen-Macaulay domain $A$ has Hilbert series $\sum_{i=0}^kh_it^i/(1-t)^d$, $d=\dim A$,
	 and the numerator is
	symmetric, i.e. $h_i=h_{k-i}$, then $A$ is Gorenstein.
\end{thm}

We will refer to $(h_0, \ldots, h_k)$ as the $h$-vector of $K[W]$. Our algebras can be realized as algebras defined by a symmetric ladder determinantal ideal. For a proof of this fact see e.\,g.\ \cite[Theorem 4.2]{co-na-pe-yu}. Conca, \cite{co}, studies the Hilbert series of such algebras. The key idea is to take an initial ideal of the ladder determinantal ideal, which then defines a Stanley-Reisner ring of a shellable simplicial complex. For Stanley-Reisner rings of shellable simplicial complexes there is a combinatorial interpretation of the $h$-vector known as the McMullen-Walkup formula. This also proves that the rings are Cohen-Macaulay. The proof by Conca applies directly also in our case, and the result is the following. 

Define an NE-path to be a lattice path in the diagram that can only go up or right (north or east). We refer to an upward step as an N-step (north), and a step to the right as an E-step (east). An NE-path is called maximal if it is of maximal length, which implies that it starts in $x_i^2$ (on the diagonal) for some $i$, and goes to $x_1x_n$ (the upper right corner). The dimension of $K[W]$
is $n$, and $h_i$ is the number of maximal paths with $i$ maximal N-parts. An N-part of a 
path is a subsequence $x_ax_b-x_{a-1}x_b-\cdots-x_{a-m}x_b$ of N-steps, and it is maximal if it can't be extended to a longer N-part of the path. The description in \cite{co} is slightly different but equivalent. 
Instead of maximal N-parts, $h_i$ is counted by the number of \emph{corners}. A corner, in this context, is the same as the starting point of a N-part. 

\begin{ex}\label{ex:h-vec}
If $W$ is as in Figure \ref{fig:firstex}, $h_0=1$, since $x_1^2-x_1x_2-x_1x_3-x_1x_4-x_1x_5-x_1x_6$ is
the only maximal path without N-parts. 
We have $h_1=7$ corresponding to the paths
\begin{align*}
	&x_2^2-x_1x_2-x_1x_3-x_1x_4-x_1x_5-x_1x_6, \ x_2^2-x_2x_3-x_1x_3-x_1x_4-x_1x_5-x_1x_6, \\
	&x_2^2-x_2x_3-x_2x_4-x_1x_4-x_1x_5-x_1x_6, \ x_2^2-x_2x_3-x_2x_4-x_2x_5-x_1x_5-x_1x_6, \\ 
	&x_2^2-x_2x_3-x_2x_4-x_2x_5-x_2x_6-x_1x_6, \ x_3^2-x_2x_3-x_1x_3-x_1x_4-x_1x_5-x_1x_6, \\
	&\mbox{and} \  x_3^2-x_3x_4-x_2x_4-x_1x_4-x_1x_5-x_1x_6. 
\end{align*}
We have $h_2=5$ corresponding to the paths
\begin{align*}
	&x_3^2-x_2x_3-x_2x_4-x_1x_4-x_1x_5-x_1x_6, \ x_3^2-x_2x_3-x_2x_4-x_2x_5-x_1x_5-x_1x_6, \\
	&x_3^2-x_2x_3-x_2x_4-x_2x_5-x_2x_6-x_1x_6, \ x_3^2-x_3x_4-x_2x_4-x_2x_5-x_1x_5-x_1x_6, \\
	&\mbox{and} \  x_3^2-x_3x_4-x_2x_4-x_2x_5-x_2x_6-x_1x_6. 
\end{align*}
There is no path with more than two maximal N-parts. Thus the Hilbert series is $(1+7t+5t^2)/(1-t)^6$.
\end{ex}

For a diagram of a strongly stable set $V$, we define a partial ordering by 
$(a_1,b_1) \le (a_2,b_2)$ if $a_1 \le a_2$ and $b_1 \le b_2$. The corners of a NE-path in the diagram constitutes an antichain (i.e. sets of points with no relation between any two) in this partial ordering. Hence $h_i$ can be determined as the number of antichains of size $i$ in the diagram of $V$ with the top row deleted. 
We state this fact as a lemma. 

\begin{lem}\label{lem:antichains}
	Let $W$ be a strongly stable set of quadratic monomials, and let $(h_1, \ldots, h_k)$ be the $h$-vector of $K[W]$. Then $h_i$ equals the number of antichains of size $i$
	 of the diagram 
	 of $W$ with the top row removed. 
\end{lem}

The antichains of length 1 always correspond to the points in the diagram with the top row
deleted. So, in Example \ref{ex:h-vec} the antichains of length 1 are in order 
$(2,2)$, $(2,3)$, $(2,4)$, $(2,5)$, $(2,6)$, $(3,3)$, and
$(3,4)$, and the antichains of length 2 are $\{(3,3),(2,4)\}$, $\{(3,3),(2,5)\}$, $\{(3,3),(2,6)\}$, $\{(3,4),(2,5)\}$, and $\{(3,4),(2,6)\}$.

\section{Gorenstein rings}
For a strongly stable set $W \subset R$ we may assume, when searching for Gorenstein rings,
that $x_1x_n$ is not the only monomial in the last column. Indeed, it follows from Lemma \ref{lem:antichains} that we 
may add or remove boxes as we like in the first row of the diagram (as long as the diagram still 
corresponds to a strongly stable set) without affecting the $h$-vector. The number of boxes in the 
first row determines the dimension of $K[W]$. As it is natural to work with lowest possible dimension, 
we will assume from now on that the first row has as few boxes as possible. This means that the last columns has at least two boxes. Or in algebraic terms, we are assuming that $x_2x_n \in W$.  We
say that $K[W]$ has no free variable.  

Let $V_{2k}=\st(x_{k+1}^2,x_kx_{k+2},x_{k-1}x_{k+3},\ldots,x_2x_{2k})$.

\begin{thm}\label{thm:V1} The ring $K[V_{2k}]$ is Gorenstein with Hilbert series $\sum_{i=0}^k \binom{k}{i}^2 t^i/(1-t)^{2k}$.
\end{thm}

\begin{proof} The diagram of $V_{2k}$ with the top row deleted is a triangle with $1+3+5+\cdots+2k-1=k^2$ points. There are $\binom{k}{i}^2$ antichains of length $i$, by \cite[Theorem 1]{pr} and
\cite[Corollary 2.4]{ste}. The result also follows from Exercise 3.47(f) in \cite{st1}. We thank 
Richard Stanley for informing us about this. Since $\sum_{i=0}^k \binom{k}{i}^2 t^i$ is symmetric, the ring is Gorenstein.
\end{proof}

As we always have $h_0=1$, a necessary condition for Gorenstein is $h_k=1$.

\begin{lem}\label{lem:necessary}
	Let $W$ be a strongly stable set of quadratic monomials, and suppose that $K[W]$ has no free variable. If $K[W]$ is Gorenstein 
	with Hilbert series $h(t)/(1-t)^n$, $deg(h(t))=k$, then $n=2k$ and $V_{2k} \subseteq W
	\subseteq {\rm st}(x_{2k}^2)$. 
\end{lem}
\begin{proof}
	An NE-path with $k$ maximal N-parts must start in $x_i^2$ for some $i>k$, so $x_{k+1}^2 \in W$. Such a path must also have at least $k-1$ E-parts, thus $n \ge 2k$. The alternating path $P=x_{k+1}^2-x_kx_{k+1}-x_kx_{k+2}-\cdots-x_2x_{2k}$ has $k$ maximal N-parts, and every other path with $k$ maximal N-parts must lie to the right of $P$. Thus, if $h_k=1$, then $P$ is the unique path with $k$ maximal N-parts. It follows that  $V_{2k}\subseteq W$ and $n=2k$, so $W\subseteq{ \st}(x_{2k}^2)$
since $V={\rm st}(x_{2k}^2)$ is the largest set with $n=2k$.
	\end{proof}

Let $D=\{x_ix_j \ | \ i+j=n+2\}$, so that if $n=2k$ the set $D$ is the set of strongly stable generators of $V_{2k}$.  

\begin{thm}\label{thm:main}
	Let $W$ be a strongly stable set of quadratic monomials. The ring $K[W]$ is Gorenstein
	if and only if for some $k$ we have $W=V_{2k}$ or $W=V_{2k} \cup \st(m_1, \dots, m_t) \subset K[x_1, x_2, \ldots, x_{2k}]$ 
	where $m_1, \ldots, m_t$ are monomials satisfying the following conditions, for any 
	$1 \le r, s \le t$. 
	
	\begin{enumerate}
		\item  $m_r=x_ix_j$ where $i \le k+1 <j$ or $i=j>k+1$.
		\item $\st(m_r) \cap \st(m_s) \cap D = \emptyset$
	\end{enumerate}
\end{thm}

Before we prove Theorem \ref{thm:main} we shall introduce the \emph{Narayana numbers} \[N(k,i)=\frac{1}{k}\binom{k}{i}\binom{k}{i-1}.\]
The Narayana number $N(k,i)$ counts the number of Dyck paths of length $2k$ with $i$ peaks. In our setting $N(k,i)$ is the number of maximal NE-paths in $V_{2k}$ with $i$ maximal N-parts starting in $x_{k+1}^2$. Notice that the Narayana numbers satisfy the relation $N(k,i)=N(k,k-i+1)$.

\begin{lem}\label{lem:V1_square}
	The ring $K[V_{2k} \cup \st(x_j^2)]$, for $k+1 \le j \le 2k$ is Gorenstein.
\end{lem}

The proof of Lemma \ref{lem:V1_square} is illustrated in Figure \ref{fig:Vsquare}.

\begin{figure}[ht]
	\includegraphics[scale=0.43]{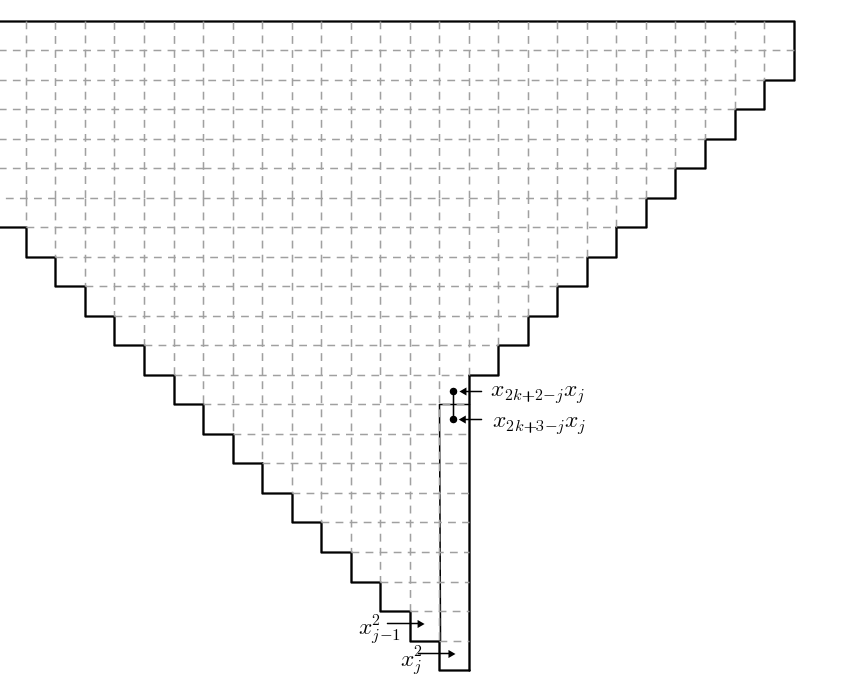}
	\caption{$V_{2k} \cup \st(x_j^2)$ in the proof of Lemma \ref{lem:V1_square}}\label{fig:Vsquare}
\end{figure}

\begin{proof}
	The proof is by induction over $j$. As $x_{k+1}^2 \in V_{2k}$ the base case $j=k+1$ is Theorem \ref{thm:V1}. We now want to compute how many maximal NE-paths with $i$ maximal N-parts are added when we extend $V_{2k} \cup \st(x_{j-1}^2)$ to $V_{2k} \cup \st(x_{j}^2)$. The paths in $V_{2k} \cup \st(x_{j}^2)$ that go outside of $V_{2k} \cup \st(x_{j-1}^2)$ are precisely those paths including the N-step $x_{2k-j+3}x_j-x_{2k-j+2}x_j$. 
	
	Let us first consider the NE-paths to $x_{2k-j+2}x_j$ where the last step is an N-step. The NE-paths from $x_\ell^2$, for any $\ell$, to $x_{2k-j+2}x_j$ are all NE-paths of length $2(j-k-1)$. We want to compute the number of such paths, ending with an N-step, and with say $i_1$ maximal N-parts not counting the last N-part. To do this we choose $2i_1+1$ numbers between $1$ and $2(j-k-1)$. The first number is the starting position of the first N-part. The second number is the starting position of the first E-part, the third number is the stating position of the second N-part, and so on, measured in length from the beginning of the path. Hence the number of choices is 
	\[\binom{2(j-k-1)}{2i_1+1}.\] 
	
	Next we consider the NE-paths from $x_{2k-j+2}x_j$ to $x_1x_{2k+1}$. Such a path must start with an N-step, which then continues the N-part in the end of the path to $x_{2k-j+2}x_j$. The number of paths from $x_{2k-j+2}x_j$ to $x_1x_{2k+1}$ with $i_2$ maximal N-parts is $N(2k-j+1,i_2)$. 
	
	We can now conclude that the number of maximal NE-paths with $i=i_1+i_2$ maximal N-parts and including the step $x_{2k-j+3}x_j-x_{2k-j+2}x_j$ is 
	\[
	\sum\binom{2(j-k-1)}{2i_1+1} N(2k-j+1,i_2)
	\] 
	 where the sum is over all $i_1, i_2$ such that $i_1+i_2=i$, $0 \le i_1 \le j-k-2$, and $1 \le i_2 \le 2k-j+1$. We get the same result if we replace $i_1$ by $j-k-2-i_1$ and $i_2$ by $2k-j+2-i_2$ since 
	 \[
	        \binom{2(j-k-1)}{2(j-k-2-i_1)+1}  =\binom{2(j-k-1)}{2i_1+1}
	 \]
	 and
	 \[
	 N(2k-j+1,2k-j+2-i_2)=N(2k-j+1,i_2).
	 \]
	 In that case we are counting the NE-paths with 
	 \[
	 (j-k-2-i_1)+(2k-j+2-i_2)=k-i
	 \]
	 maximal N-steps, so we can conclude that the $h$-vector is symmetric.
\end{proof}

We are now ready to prove that the conditions in Theorem \ref{thm:main} are sufficient. The proof is illustrated in Figure \ref{fig:if}.

\begin{figure}[ht]
	\includegraphics[scale=0.47]{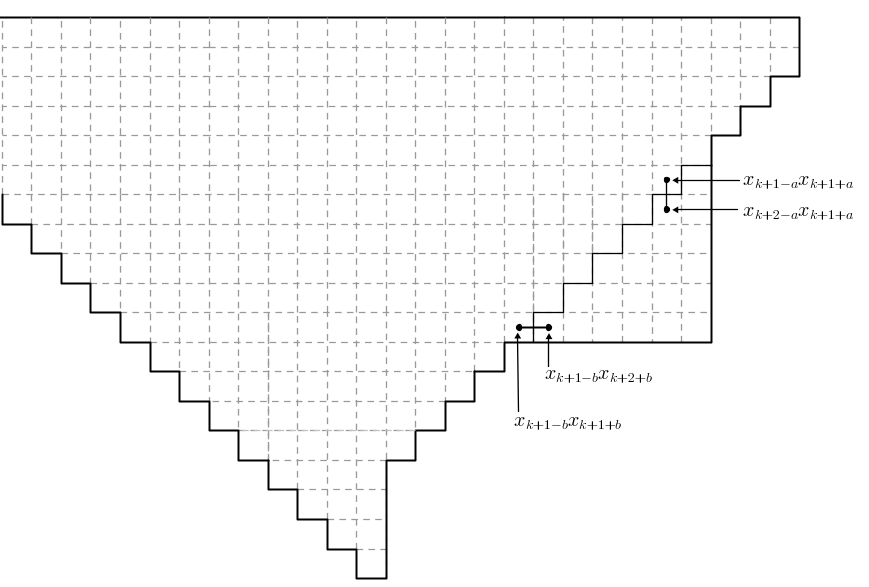}
	\caption{The fixed steps in the ``if''-part of Theorem \ref{thm:main}.}
	\label{fig:if}
\end{figure} 

\begin{proof}[Proof of the ``if''-part of Theorem \ref{thm:main}]
  We shall now prove that a strongly stable set 
  \[W=V_{2k} \cup \st(m_1, \dots, m_t)\]
   with $m_1, \ldots, m_t$ as stated in Theorem \ref{thm:main} defines a Gorenstein ring. Notice that at most one of $m_1, \ldots, m_t$ can be a square, as $\st(x_i^2) \subset \st(x_j^2)$ if $i<j$. If $t=1$ and $m_1=x_j^2$ the result is proved in Lemma \ref{lem:V1_square}. We assume from now on that $t>1$ or $t=1$ but $m_1$ is not a square. We proceed by induction over $t$. We may order $m_1, \ldots, m_t$ so that $m_t$ sits above and to the right of $m_1, \dots, m_{t-1}$ in the diagram. 
  
  Suppose the strongly stable set $W'=V_{2k} \cup \st(m_1, \dots, m_{t-1})$ defines a Gorenstein ring. If $t=1$ this should be interpreted as $W'=V_{2k}$, which indeed is Gorenstein. We want to prove that when we extend $W'$ to $W' \cup \st(m_t)=W$ the number of new paths with $i$ N-parts and the number of new paths with $k-i$ N-parts are the same. 
  
  Let us fix an E-step $x_{k+1-b}x_{k+1+b}-x_{k+1-b}x_{k+2+b}$ and an N-step $x_{k+2-a}x_{k+1+a}-x_{k+1-a}x_{k+1+a}$ inside $\st(m_t)$, such that $a>b$. The idea is to consider paths that stays inside $W'$ before the fixed E-step and after the fixed N-step. Every path in $W$ that goes outside $W'$ can be described in this way for a unique pair $a,b$. Hence is it enough to prove the statement for a fixed $a$ and $b$. 
  
  Let us first consider paths to $x_{k+1-b}x_{k+1+b}$ in $W'$ with $i_1$ maximal N-parts. Let us call the number of such paths $\beta(i_1)$. This is the same as the number of paths with $i_1$ maximal N-parts in $V_{2b} \cup \st(m_1, \ldots, m_{t-1})$ which defines a Gorenstein ring, by the induction hypothesis. It follows that $\beta(i_1)=\beta(b-i_1)$. 
  
  Next we consider paths from $x_{k+1-b}x_{k+2+b}$ to $x_{k+2-a}x_{k+1+a}$ with $i_2$ maximal N-parts. If the path ends with an N-step we will not count the last N-part, as it will be counted in the next step of the proof. These paths stay inside a square with side $a-b$. To do this we choose $i_2$ points that represent the endpoints of the maximal N-parts. We do not choose points from the rightmost column or the bottom row in the square. There are $\binom{a-b-1}{i_2}^2$ ways to choose these points. 
  
  Last we consider paths from $x_{k+1-a}x_{k+1+a}$ to $x_2x_{2k}$ with $i_3$ maximal N-parts. The number of such paths is $N(k-a,i_3)$. We can now conclude that the number of maximal NE-paths with $i$ N-parts and a fixed E-step $x_{k+1-b}x_{k+1+b}-x_{k+1-b}x_{k+2+b}$ and N-step $x_{k+2-a}x_{k+1+a}-x_{k+1-a}x_{k+1+a}$ is
  \[
  \sum \beta(i_1) \binom{a-b-1}{i_2}^2N(k-a,i_3)
  \] 
  where the sum is over all $i_1, i_2, i_3$ such that $i_1+i_2+i_3=1$, $0 \le i_1 \le b$, $0 \le i_2 \le a-b-1$ and $1 \le i_3 \le k-a $. Notice that we obtain the same result if we replace $i_1$ by $b-i_1$, $i_2$ by $a-b-1-i_2$, and $i_3$ by $k-a+1-i_3$. In that case we are counting the paths with 
  \[
  (b-i_1)+(a-b-1-i_2)+(k-a+1-i_3)=k-i
  \] 
  maximal N-parts. As this holds for all possible choices of $a$ and $b$, we have now proved that  
   the number of new paths with $i$ N-parts and the number of new paths with $k-i$ N-parts are the same, when we extend $W'$ to $W$.
\end{proof}

Before we prove the "only if" part, we give an example. Suppose $k=5$ and that
${\rm st}(m_1,\ldots,m_t)\cap D=\{x_3x_9,x_4x_8,x_5x_7\}$, so $\{m_1,\ldots,m_t\}$ is either
$\{x_4x_9,x_5x_8\}$ or $\{x_5x_9\}$. We will look for a condition for $h_1=h_{k-1}=h_4$. The
paths with four maximal N-parts which are not paths in $V_{10}$ all starts with $x_6^2-x_5x_6-x_5x_7$ and all ends with $x_3x_9-x_2x_9-x_2x_{10}-x_1x_{10}$. There are three possibilities
between $x_5x_7$ and $x_3x_9$, namely $x_5x_7- x_5x_8 -x_4x_8-x_4x_9 -x_3x_9$, $x_5x_7-x_4x_7-x_4x_8-x_4x_9  -x_3x_9$, or $x_5x_7-x_5x_8-x_4x_8-x_3x_8  -x_3x_9$. Thus $h_4$ increases with three when we extend $V_{10}$ by $\st(x_4x_9,x_5x_8)$ or $\st(x_5x_9)$. The increase
in $h_1$ is either two or three for the two possibilities.

\begin{proof}[Proof of the ``only if''-part of Theorem \ref{thm:main}]
By Lemma \ref{lem:necessary} we may assume $n=2k$ and $V_{2k} \subseteq W$.  
The idea is to prove that $h_{k-1}>h_1$ when the given conditions are not satisfied. To this end, we start by describing all possible paths with $k-1$ maximal N-parts. Such paths have the three possible starting points $x_k^2$, $x_{k+1}^2$ and $x_{k+2}^2$.
In addition to $D$ as previously defined, we let $D'=\{x_ix_j \ | \ i+j=2k+3\}$ and $D''=\{x_ix_j \ | \ i+j=2k+4\}$. The paths starting in $x_k^2$ have all $k-1$ maximal N-parts of length 1, and stay inside $V_{2k}$.  The paths starting in $x_{k+1}^2$ have one maximal N-part of length 2, and the remaining $k-2$ maximal N-parts of length 1. Such a path lives in $V_{2k} \cup D'$, and after the N-part of length 2 it stays inside $V_{2k}$.  The paths starting in $x_{k+2}^2$ can have two maximal N-parts of length 2, and $k-3$ maximal N-parts of length 1, or it can have one N-part of length 3, and $k-2$ maximal N-parts of length 1. Such a path lives in $V_{2k} \cup D' \cup D''$. After an N-part of length greater than 1 it stays inside $V_{2k} \cup D'$, and after an N-part of length 3, or two N-parts of length 2 it stays inside $V_{2k}$. Notice also that the paths starting in $x_{k+2}^2$ have all their maximal E-parts of length 1.

Let us partition the set of monomials $\{m_1, \ldots, m_t\}$ into subsets $M_1, \ldots, M_s$, and let $D_r=\st(M_r) \cap D$ for $r=1, \ldots, s$. We choose the finest possible partition such that $D_1, \ldots, D_s$ are disjoint, and each $D_r$ is connected in the sense that $D_r=\{x_ix_j \ | \ i+j=2k+2, a_r \le i \le b_r\}$ for some numbers $a_r$ and $b_r$. The conditions of Theorem \ref{thm:main} are fulfilled precisely when each $M_r$ consists of one monomial $x_ix_j$ where $i \le k+1 <j$ or $i=j>k+1$. Notice that a NE-path with $k-1$ maximal N-parts can enter at most one of the areas $\st(M_r)\setminus V_{2k}$. The same holds for paths with one maximal N-part. It is therefore enough to prove $h_{k-1}>h_1$ for each $M_r$ separately.  

Let us first consider the case when $M_r$ contains no monomial $x_ix_j$ with $i>k+1$. We shall compute the increase in $h_1$ and $h_{k-1}$ when we extend $V_{2k}$ to $V_{2k} \cup \st(M_r)$. The increase in $h_1$ is the number of boxes added, which is $\le \binom{b_r-a_r+1}{2}$ with equality precisely if $M_r$ consists of one monomial. The increase in $h_{k-1}$ is the number of paths that starts in $x_{k+1}^2$ and alternates N-steps and E-steps, except one E-part of length two where the path enters $D'$, and one N-part of length two where it leaves $D'$ for good. There are $b_r-a_r$ points in $D' \cap \st(M_r)$. As the point where the path enters $D'$ and the last point it visits on $D'$ may be the same point the increase in $h_{k-1}$ is  $\binom{b_r-a_r+1}{2}$. As $K[V_{2k}]$ is Gorenstein, we can conclude that $h_{k-1}>h_1$ for $V_{2k} \cup \st(M_r)$, if $M_r$ consists of more than one monomial. 

Next we consider the case when $M_r$ contains a monomial $x_ix_j$ with $i>k+1$. The idea in this case is to extend $V_{2k} \cup \st(M_r)$ to a strongly stable set that defines a Gorenstein algebra, and prove that the increase in $h_1$ is higher than the increase in $h_{k-1}$. As a first step, we extend $V_{2k} \cup\st(M_r)$ to $V_{2k} \cup\st(M_r) \cup \st(x_\ell^2)$, where $\ell$ is the greatest integer for which  we are not adding any boxes on $D''$. Hence all new boxes added when we go from   $V_{2k} \cup\st(M_r)$ to $V_{2k} \cup\st(M_r) \cup \st(x_\ell^2)$ lie below $D''$, so this does not increase $h_{k-1}$. If $\st(M_r) \subseteq \st(x_\ell^2)$ we can conclude $h_{k-1}>h_1$ for $K[V_{2k} \cup \st(M_r)]$, unless of course $M_r=\{x_\ell^2\}$. Suppose   $\st(M_r) \not \subseteq \st(x_\ell^2)$. Then, as the next step, we extend $V_{2k} \cup\st(M_r) \cup \st(x_\ell^2)$ to $V_{2k}\cup \st(x_\ell^2, x_ix_j)$ where $i=2k+2-\ell$ and $j$ is the smallest index for which $\st(M_r) \subseteq \st(x_\ell^2,x_ix_j)$. Notice that the monomial $x_{i+1}x_{\ell+1}$ on $D''$ is not included. This means that a path with $k-1$ maximal N-parts must stay inside $V_{2k} \cup D'$ after the $\ell$-th column. As all boxes added when we extend  $V_{2k} \cup\st(M_r) \cup \st(x_\ell^2)$ to $V_{2k}\cup \st(x_\ell^2, x_ix_j)$ lie to the right of $D'$ we do not increase $h_{k-1}$. As the final step we shall now extend $V_{2k}\cup \st(x_\ell^2, x_ix_j)$ to $V_{2k}\cup \st(x_j^2)$. The increase in $h_1$ in this step is the number of boxes added, which is 
\[
(j-\ell) \frac{(\ell-i+1)+(j-i)}{2} = (j-\ell) \frac{\ell+j-2i+1}{2}.
\] 
The new paths with $k-1$ maximal N-parts allowed by this extension are precisely those that go through $x_{i+1}x_{\ell+1}$. To reach this box we must start in $x_{k+1}^2$ and alternate N-steps and E-steps. If we continue to alternate N-steps and E-steps after $x_{i+1}x_{\ell+1}$ we end up in $x_{2k+3-j}x_{j+1}$. This is illegal, because $x_{2k+3-j}x_{j+1}$ is two steps below our diagram. To solve this we must insert our two extra N-steps in the columns $\ell+1$ through $j$. This can be done in $\binom{j-\ell+1}{2}$ ways, which is then the increase in $h_{k-1}$. To conclude that the increase in $h_1$ was higher, first note that $\ell>i$ as $i= 2k+2-\ell$ and $\ell>k+1$. The increase in $h_1$ was
\[
(j-\ell) \frac{\ell+j-2i+1}{2} >(j-\ell) \frac{j-\ell+1}{2} = \binom{j-\ell+1}{2}.
\]
To summarize this case we have extended $V_{2k} \cup \st(M_r)$ to $V_{2k} \cup \st(x_j^2)$ which defines a Gorenstein algebra. Unless both sets are equal the increase in $h_1$ was higher than the increase in $h_{k-1}$, which implies $h_{k-1}>h_1$ for $K[V_{2k} \cup \st(M_r)]$. We can conclude that the only Gorenstein algebras are those characterized by the conditions in Theorem \ref{thm:main}.
\end{proof}

\begin{cor}
For a strongly stable set $W$ of quadratic monomials, the ring $K[W]$ is Gorenstein if and only if $h_i=h_{k-i}$ for $i=0$ and $i=1$.
\end{cor}

\begin{cor}
Let $W$ be a strongly stable set of quadratic monomials.
If $K[W]$ is Gorenstein and has no free variable, then $\dim(K[W])$ is even.
\end{cor}

\begin{ex}
In \cite{mi-na} the authors construct Gorenstein algebras with $h$-vectors $(1,s+t,st+2,s+t,1)$
for each $s,t\ge2$ and with $h$-vectors $(1,s+t,st+t+1,st+t+1,s+t,1)$ for each $s\ge2$ and $t\ge3$.
In Question 2.12 they ask if these are the only $h$-vectors for Artinian Gorenstein rings of
socle degrees 4 and 5. We have calculated the $h$-vectors for our rings of dimension $2k$ with $k\le5$ in the Appendix. We can factor out with a linear regular sequence to get Artinian rings with the same $h$-vector and socle degree $k$. In fact, all our examples with $k=4$ or $k=5$ give counterexamples to their question.
\end{ex}

We have seen that if $K[W]$ is Gorenstein with no free variable and has Hilbert series 
$h(t)/(1-t)^{\dim K[W]}$, where $\deg(h(t))=k$, then $\dim K[W]=2k$ and $W$ lies between
$V_{2k}$ and $\st(x_{2k}^2)$. Both extreme sets give Gorenstein rings by
Theorem \ref{thm:main}. In fact, $K[\st(x_{2k}^2)]$ is the second Veronese subring of
$K[x_1,\ldots,x_{2k}]$, and it was proved to be Gorenstein in \cite[Theorem 3.2.1]{go-wa}. The Hilbert series of $K[\st(x_n^d)]$ was computed for $d=2$ in \cite{co} and for general $d$ in \cite{br-we}. For $d=2$ one can easily verify that the $h$-vector is symmetric if and only if $n$ is even. 

\begin{thm}\label{thm:veronese}
$K[\st(x_n^2)]$ has Hilbert series $\sum_{i=1}^{\lfloor n/2\rfloor}\binom{n}{2i}t^i/(1-t)^n$.
It is Gorenstein if and only if $n$ is even.
\end{thm}
%

\begin{cor}
In the triangular area consisting of the integer points in $\{ (x,y);y\ge1,1\le x\le n-1,y\le x\}$ with
the partial order $(a,b)\le(c,d)$ if $a\le c$ and $b\le d$, there are $\binom{n}{2i}$ antichains
of size $i$, $0\le i\le n/2$.
\end{cor}

\begin{proof}
Consider the picture for $\st(x_n^2)$ with the top row deleted. If we move this one step to the 
north and one step to the west we get the same number of antichains of each size. The
Hilbert series of $K[\st(x_n^2)]$ is $\sum_{0\le i\le n/2}\binom{n}{2i}t^i/(1-t)^n$, \cite{br-we}.
Then use Lemma \ref{lem:antichains}.
\end{proof}

We will finish by computing the Hilbert series of two more classes of Gorenstein rings from Theorem \ref{thm:main}. 

\begin{thm}
	The ring $K[\st(x_2x_{2k},x_{2k-1}^2)]$ is Gorenstein with Hilbert series
	\[
	\frac{\displaystyle \sum_{i=0}^{k}\left(\binom{2k-1}{2i}+\binom{2k-2}{2(i-1)}\right)t^i}{(1-t)^{2k}}.
	\]
\end{thm}
\begin{proof}
	All maximal NE-paths in the diagram of $\st(x_2x_{2k},x_{2k-1}^2)$ end with a step E or steps EN. The paths ending with an E can be seen as paths in the diagram of $\st(x_{2k-1}^2)$ followed by an E.   The paths ending with EN can be seen as paths in the diagram of $\st(x_{2k-2}^2)$ followed by  EN. The $h$-vector now follows from Theorem \ref{thm:veronese}.
\end{proof}

The strongly stable sets in the next theorem are precisely those obtained from $V_{2k}$ by adding one monomial. 
\begin{thm}
	The ring $K[V_{2k} \cup \st(x_ax_{2k+3-a})]$, with $3 \le a \le k+1$, is Gorenstein with Hilbert series
	\[
	\frac{\displaystyle \sum_{i=0}^k \binom{k}{i}^2t^i + \frac{1}{a-2}\left( \sum_{j=0}^{k-a+1} \binom{k-a+1}{j}^2t^j \right)\left(\sum_{m=0}^{a-2}\binom{a-2}{m}\binom{a-2}{m-1}t^m   \right)}{(1-t)^{2k}}.
	\]
\end{thm}
\begin{proof}
	Let $W=V_{2k} \cup \st(x_ax_{2k+3-a})$. As mentioned in the paragraph before the theorem $W=V_{2k} \cup \{x_ax_{2k+3-a}\}$. A maximal NE-path in $W$ that goes outside of $V_{2k}$ must contain the EN steps  $x_ax_{2k+2+a} - x_ax_{2k+3-a} - x_{a-1}x_{2k+3-a}$. Hence an NE-path that goes through $x_ax_{2k+3-a}$ consists of a path $P$ from some $x_b^2$ to $x_ax_{2k+2+a} $  and a path $Q$ from $ x_{a-1}x_{2k+3-a}$ to $x_1x_{2k}$. If $P$ has $j$ maximal N-parts and $Q$ has $m$ maximal N-parts then the whole path has $j+m$ maximal N-parts. The path $P$ can be considered as a path in $V_{2(k-a+1)}$. The number of choices for $Q$ with $m$ maximal N-parts is given by the Narayana number $N(a-2,m)$. Applying Theorem \ref{thm:V1} we get the numerator of the Hilbert series of $K[V_{2k} \cup \st(x_ax_{2k+3-a})]$ as 
	\[
	\sum_{i=0}^k \binom{k}{i}^2t^i + \left( \sum_{j=0}^{k-a+1} \binom{k-a+1}{j}^2t^j \right)\left(\sum_{m=0}^{a-2}N(a-2,m)t^m   \right). \qedhere
	\] 
\end{proof}

\bigskip

\subsection*{Acknowledgement}
We would like to thank the anonymous referee for careful reading.

\pagebreak
\appendix 
\section{Gorenstein rings of dimension $2k$, $k \le 5$}
In the below table we list the all Gorenstein rings with no free variable generated by strongly stable sets of quadratic monomials, up to dimension 10. The rings are defined as $A=K[V_{2k} \cup W]$, where $W$ is listed in the table.   
 
\noindent \begin{tabular}{|c|c|c|}
	\hline
$k$&$h$-vector $A$ &$W$\\
\hline
1&(1, 1)&$\emptyset$\\
2&(1, 4, 1)&$\emptyset$\\
2&(1, 5, 1)&${\rm st}(x_3x_4)$\\
2&(1, 6, 1)&${\rm st}(x_4^2)$\\
3&(1, 9, 9, 1)&$\emptyset$\\
3&(1, 10, 10, 1)&${\rm st}(x_4x_5)$\\
3&(1, 10, 10, 1)&${\rm st}(x_3x_6)$\\
3&(1, 12, 12, 1)&${\rm st}(x_4x_6)$\\
3&(1, 11, 11, 1)&${\rm st}(x_5^2)$\\
3&(1, 15, 15, 1)&${\rm st}(x_6^2)$\\
4&(1, 16, 36, 16, 1)&$\emptyset$\\
4&(1, 17, 39, 17, 1)&${\rm st}(x_5x_6)$\\
4&(1, 17, 38, 17, 1)&${\rm st}(x_4x_7)$\\
4&(1, 17, 40, 17, 1)&${\rm st}(x_3x_8)$\\
4&(1, 18, 44, 18, 1)&${\rm st}(x_5x_6, x_3x_8)$\\
4&(1, 19, 43, 19, 1)&${\rm st}(x_5x_7)$\\
4&(1, 19, 44, 19, 1)&${\rm st}(x_4x_8)$\\
4&(1, 18, 42, 18, 1)&${\rm st}(x_6^2)$\\
4&(1, 19, 48, 19, 1)&${\rm st}(x_6^2, x_3x_8)$\\
4&(1, 22, 50, 22, 1)&${\rm st}(x_7^2)$\\
4&(1, 28, 70, 28, 1)&${\rm st}(x_8^2)$\\
& & \\
& & \\
\hline
\end{tabular}
\begin{tabular}{|c|c|c|}
	\hline
	$k$&$h$-vector $A$&$W$\\
	\hline
5&(1, 25, 100, 100, 25, 1)&$\emptyset$\\
5&(1, 26, 106, 106, 26, 1)&${\rm st}(x_6x_7)$\\
5&(1, 26, 104, 104, 26, 1)&${\rm st}(x_5x_8)$\\
5&(1, 26, 105, 105, 26, 1)&${\rm st}(x_4x_9)$\\
5&(1, 26, 109, 109, 26, 1)&${\rm st}(x_3x_{10})$\\
5&(1, 27, 112, 112, 27, 1)&${\rm st}(x_6x_7, x_4x_9)$\\
5&(1, 27, 116, 116, 27, 1)&${\rm st}(x_6x_7, x_3x_{10})$\\
5&(1, 27, 114, 114, 27, 1)&${\rm st}(x_5x_8, x_3x_{10})$\\
5&(1, 28, 114, 114, 28, 1)&${\rm st}(x_6x_8)$\\
5&(1, 29, 126, 126, 29, 1)&${\rm st}(x_6x_8, x_3x_{10})$\\
5&(1, 29, 128, 128, 29, 1)&${\rm st}(x_4x_8)$\\
5&(1, 28, 119, 119, 28, 1)&${\rm st}(x_4x_{10}, x_6x_7)$\\
5&(1, 31, 127, 127, 31, 1)&${\rm st}(x_6x_9)$\\
5&(1, 31, 131, 131, 31, 1)&${\rm st}(x_5x_{10})$\\
5&(1, 35, 135, 135, 35, 1)&${\rm st}(x_6x_{10})$\\
5&(1, 27, 112, 112, 27, 1)&${\rm st}(x_7^2)$\\
5&(1, 28, 119, 119, 28, 1)&${\rm st}(x_7^2, x_4x_9)$\\
5&(1, 28, 123, 123, 28, 1)&${\rm st}(x_7^2, x_3x_{10})$\\
5&(1, 30, 137, 137, 30, 1)&${\rm st}(x_7^2, x_4x_{10})$\\
5&(1, 31, 128, 128, 31, 1)&${\rm st}(x_8^2)$\\
5&(1, 32, 143, 143, 32, 1)&${\rm st}(x_8^2, x_3x_{10})$\\
5&(1, 37, 154, 154, 37, 1)&${\rm st}(x_9^2)$\\
5&(1, 45, 210, 210, 45, 1)&${\rm st}(x_{10}^2)$\\
\hline
\end{tabular}


\newpage

\begin{thebibliography}{}

\bibitem{bo-co} M. Boij and A. Conca, {\it On the Fr\"oberg-Macaulay conjectures for algebras},
Rend. Istit. Mat. Univ. Trieste, {\bf 50} 139--147 (2018)

\bibitem{br-we}  F. Brenti and V. Welker, {\it The Veronese construction for formal power
series and graded algebras},  Adv. Appl. Math., {\bf 42} (4) 545-556 (2009)   
 
 \bibitem{co} A. Conca, {\it Symmetric ladders}, Nagoya Math. J., {\bf 136}, 35--56 (1994)
 

 
 \bibitem{co-na-pe-yu} A. Corso, U. Nagel, S. Petrovi\'{c}, and C. Yen, {\it Blow-Up Algebras,
 Determinental ideals, and Dedekind-Mertens-Like Formulas}, Forum Mathematicum, {\bf 29} (4)
 799--830 (2017)
 
 \bibitem{dene} E. De Negri, {\it Toric rings generated by special stable sets of monomials}, Math. Nachr. {\bf 203} 31--45 (1999) 
 
  \bibitem{ga} A. Galligo, {\it Th\'eor\`eme de division et stabilit\'e en g\'eometrie analytique locale}, Ann. Inst. Fourier (Grenoble), {\bf 29}(2):vii 107--184 (1979)

 \bibitem{go-wa} S. Goto and K. Watanabe, {\it On graded rings I}, J. Math. Soc. Japan {\bf 30}:2
 (1978)
 
\bibitem{he-hi} J. Herzog and T. Hibi, {\it Monomial ideals}, Grad. Texts in Math. {\bf 260} Springer (2011)

\bibitem{le} P. Lella, {\it An efficient implementation of the algorithm computing the Borel-fixed 
points of a Hilbert scheme}, ISSAC 2012 — Proceedings of the 37th International Symposium on 
Symbolic and Algebraic Computation, 242--248 ACM, New York (2012)

\bibitem{mi-na} J. Migliore and U. Nagel, {\it Gorenstein algebras presented by quadrics},
Collect. Math. {\bf 64} 211--233 (2013)

\bibitem{ni} L. Nicklasson, {\it Subalgebras generated in degree two with minimal Hilbert function}, Math. Scand. {\bf 127} 5--27 (2021)

\bibitem{no-sp} R. Notari and M. L. Spreafico, {\it A stratification of Hilbert schemes by initial 
ideals and applications}, Manuscripta Math. {\bf 101} (4) 429--448 (2000) 

 \bibitem{pr} R. A. Proctor, {\it Shifted plane partitions of trapezoidal shape}, Proc. Am. Math. Soc. {\bf 89} 553--559 (1983)
 
\bibitem{st} R. Stanley, {\it Hilbert functions of graded algebras}, Adv. in Math. {\bf 28} 57--83 (1978)

\bibitem{st1} R. Stanley, {\it Enumerative combinatorics} vol. 1, $2^{nd}$ ed., Cambridge Univ.
Press (2011)

\bibitem{ste} J.~R. Stembridge. {\it Trapezoidal Chains and antichains}, Europ. J. Combinatorics
{\bf 7} 377--387 (1986)

\end{thebibliography}
\end{document}